\documentclass[12pt]{amsart}

\usepackage{mathrsfs, amssymb,amsthm, amsfonts, amsmath}
\usepackage[all]{xy}
\usepackage{upgreek}
\usepackage{amsmath}

\newtheorem{theorem}[equation]{Theorem}
\newtheorem*{theorem*}{Theorem}
\newtheorem{lemma}[equation]{Lemma}
\newtheorem{corollary}[equation]{Corollary}

\newtheorem{proposition}[equation]{Proposition}

\theoremstyle{definition}
\newtheorem{example}[equation]{Example}

\newtheorem*{definition*}{Definition}

\theoremstyle{remark}
\newtheorem{remark}[equation]{Remark}

\makeatletter\@addtoreset{equation}{section}
\makeatother

\newcommand{\QQ}{\mathbb{Q}}

\newcommand{\PP}{\mathbb{P}}
\newcommand{\KK}{\mathbb{K}}
\newcommand{\LL}{\mathbb{L}}

\newcommand{\LLL}{{\mathscr{L}}}
\newcommand{\HHH}{{\mathscr{H}}}
\newcommand{\SSS}{\mathfrak{S}}

\newcommand{\mumu}{{\boldsymbol{\mu}}}

\newcommand{\Aut}{\operatorname{Aut}}
\newcommand{\Bir}{\operatorname{Bir}}

\newcommand{\GL}{\operatorname{GL}}
\newcommand{\Gal}{\operatorname{Gal}}

\newcommand{\Fix}{\operatorname{Fix}}
\newcommand{\z}{\operatorname{z}}
\newcommand{\rk}{\operatorname{rk}}
\newcommand{\mult}{\operatorname{mult}}
\newcommand{\Pic}{\operatorname{Pic}}
\newcommand{\rkPic}{\rk\Pic}
\newcommand{\WE}{\mathrm{W}(\mathrm{E}_6)}
\newcommand{\GAff}{\mathrm{GA}_2(\mathbf{F}_3)}

\def \ge {\geqslant}
\def \le {\leqslant}

\date{}

\title{Automorphisms of cubic surfaces without points}

\author{Constantin Shramov}

\address{Steklov Mathematical Institute of Russian Academy of Sciences, 8 Gubkina st.,
Moscow, 119991, Russia
}

\email{costya.shramov@gmail.com}

\thanks{This work is supported by the Russian Science Foundation
under grant No.~19-11-00237.}

\begin{document}

\begin{abstract}
We classify finite groups acting by birational transformations
of a non-trivial Severi--Brauer surface over a field of characteristc zero that
are not conjugate to subgroups of the automorphism group.
Also, we show that the automorphism group of a smooth cubic
surface over a field~$\KK$ of characteristic zero
that has no $\KK$-points is abelian, and
find a sharp bound for the Jordan constants of birational automorphism groups
of such cubic surfaces.
\end{abstract}

\maketitle

\section{Introduction}

Given a variety $X$, it is natural to try to describe finite subgroups
of its birational automorphism group $\Bir(X)$ in terms of the birational
models of $X$ on which these finite subgroups are regularized.
In dimension $2$, this is sometimes possible due to the well developed Minimal Model
Program. In the case of the group of birational automorphisms of the projective
plane over an algebraically closed field of characteristic zero,
this was done in the work of I.\,Dolgachev and V.\,Iskovskikh~\cite{DI}.
Over algebraically non-closed fields, some partial results are also known
for non-trivial Severi--Brauer surfaces, i.e.
del Pezzo surfaces of degree $9$ not isomorphic to $\PP^2$ over the base field.
Note that the latter surfaces are exactly the del Pezzo surfaces
of degree $9$ without points over the base field,
see e.g.~\cite[Theorem~53(2)]{Kollar-SB}.

Recall that a birational map $Y\dasharrow X$ defines an embedding of
a group $\Aut(Y)$ into~\mbox{$\Bir(X)$}.
We will denote by $\rkPic(X)^{\Gamma}$ the invariant part of the Picard group
$\Pic(X)$ with respect to the action of a group~$\Gamma$, and
by $\mumu_n$ the cyclic group of
order~$n$. The following partial description of
finite groups of birational automorphisms of non-trivial Severi--Brauer
surfaces is known.

\begin{proposition}[{\cite[Proposition~3.7]{Shramov-SB},
\cite[Corollary~4.5]{Shramov-SB}}]
\label{proposition:summary}
Let $S$ be a non-trivial Severi--Brauer surface over a field $\KK$ of characteristic zero, and let
$G\subset\Bir(S)$ be a finite subgroup.
Then $G$ is conjugate either to a subgroup of
$\Aut(S)$,  or to a subgroup of~\mbox{$\Aut(S')$},
where $S'$ is a smooth cubic surface over $\KK$ birational to $S$
such that~\mbox{$\rkPic(S')=3$} and~\mbox{$\rkPic(S')^G=1$}.
In the latter case $G$ is isomorphic to
a subgroup of~$\mumu_3^3$.
\end{proposition}

However, it was not known whether the second of the two cases listed in
Proposition~\ref{proposition:summary} can indeed occur.
The first goal of this paper is to construct such examples and thus
complete the classification of finite
subgroups of birational automorphism groups
of non-trivial Severi--Brauer surfaces that are not conjugate to subgroups
of automorphism groups.

\begin{theorem}\label{theorem:main}
Let $G$ be one of the groups~$\mumu_3$,~$\mumu_3^2$,
and~$\mumu_3^3$.
Then there exists a field
$\KK$ of characteristic zero and a non-trivial Severi--Brauer surface
$S$ over $\KK$ such that $\Bir(S)$ contains a subgroup isomorphic to $G$
and not conjugate to a subgroup of~\mbox{$\Aut(S)$}.
\end{theorem}

Proposition~\ref{proposition:summary}
was used in \cite{Shramov-SB} to study finite
subgroups of birational automorphism groups  of non-trivial Severi--Brauer
surfaces.

\begin{theorem}[{\cite[Theorem~1.2(ii)]{Shramov-SB}, \cite{Shramov-SB-short}}]
\label{theorem:SB}
Let $\KK$ be a field of characteristic zero, and let $S$ be a
non-trivial Severi--Brauer surface over $\KK$.
Let $G$ be a finite subgroup of~\mbox{$\Bir(S)$}. Then $G$ is either abelian,
or contains a normal abelian subgroup of
index~$3$. Furthermore, there exists a field $\KK_0$ of characteristic
zero, a non-trivial Severi--Brauer surface $S_0$ over $\KK_0$,
and a finite subgroup $G_0\subset\Aut(S_0)$ such that
the minimal index of a normal abelian subgroup in $G_0$ equals~$3$.
\end{theorem}

Using the notion of \emph{Jordan constant}
(see \cite[Definition~1]{Popov14}), one can reformulate
Theorem~\ref{theorem:SB} by saying that
the Jordan constant of the birational automorphism group
of a non-trivial Severi--Brauer surface over a field of characteristic zero
does not exceed $3$, and this bound is attained over suitable fields.
We refer the reader to~\cite[Theorem~1.9]{Yasinsky} for an analog of this
result for the group of birational automorphisms of the projective
plane.

As a by-product of the constructions used in this paper, we prove an
analog of Theorem~\ref{theorem:SB} for automorphism
groups of cubic surfaces without points
over the base field
(including those that are not birational to any Severi--Brauer surface).

\begin{theorem}\label{theorem:cubic}
Let $\KK$ be a field of characteristic zero, and let $S$ be a smooth cubic
surface over $\KK$. Suppose that $S$ has no $\KK$-points.
The following assertions hold.
\begin{itemize}
\item[(i)] The group $\Aut(S)$ is abelian.

\item[(ii)] Every finite subgroup of $\Bir(S)$ has a normal
abelian subgroup of index at most $3$.

\item[(iii)] There exists a field $\KK_0$ of characteristic
zero, a smooth cubic surface $S_0$ over $\KK_0$
with~\mbox{$S_0(\KK_0)=\varnothing$},
and a finite subgroup $G_0\subset\Bir(S_0)$ such that
the minimal index of a normal abelian subgroup in $G_0$ equals~$3$.
\end{itemize}
\end{theorem}

In other words, if $S$ is a smooth cubic surface
over a field $\KK$ of characterisitic zero with~\mbox{$S(\KK)=\varnothing$},
then the Jordan constant
the group $\Bir(S)$ is at most~$3$, and this bound is attained for certain
cubic surfaces over suitable fields.

In course of the proof of our main results, we obtain a classification
of birational models of smooth cubic surfaces
without points and del Pezzo surfaces of degree~$6$
without points of degree~$1$ and~$2$
that we find interesting on its own.

\begin{proposition}\label{proposition:all-birational-models}
Let $\KK$ be a field of characteristic zero, and let $S$ be a surface
over $\KK$ of one of the following three types:
\begin{itemize}
\item a smooth cubic surface with no $\KK$-points;

\item a del Pezzo surface of degree $6$ with no points
of degree $1$ and $2$ over $\KK$;

\item a non-trivial Severi--Brauer surface.
\end{itemize}
Suppose that $S'$ is a surface over $\KK$ birational to $S$,
such that $S'$ is either a del Pezzo
surface or a conic bundle. Then $S'$ is also a surface of one of the above
three types. In particular, $S'$ is not a conic bundle.
\end{proposition}

\smallskip
The plan of the paper is as follows.
In Section~\ref{section:cubic} we prove some auxiliary results
concerning birational geometry of cubic surfaces.
In Section~\ref{section:main-example}
we provide a construction that proves
Theorem~\ref{theorem:main}.
In Section~\ref{section:plane}
we discuss automorphisms of plane cubic curves.
In Section~\ref{section:Jordan}
we prove the first assertion of Theorem~\ref{theorem:cubic}.
In Section~\ref{section:dP6}
we make some additional observations on del Pezzo surfaces of degree $6$.
In Section~\ref{section:Bir}
we complete the proof of Theorem~\ref{theorem:cubic}.

\smallskip
\textbf{Notation and conventions.}
We denote by $\SSS_n$ the symmetric group on $n$ letters.

Given a field $\KK$, we denote by $\bar{\KK}$ its algebraic closure.
If $X$ is a variety defined over~$\KK$ and~$\LL$ is an extension of~$\KK$,
we denote by~$X_{\LL}$
the extension of scalars of $X$ to~$\LL$.
By a point of degree $r$ on a variety defined over some field $\KK$ we mean a closed point whose
residue field is an extension of~$\KK$ of degree~$r$;
a $\KK$-point is a point of degree~$1$.
The set of all $\KK$-points of $X$ is denoted by~$X(\KK)$.

A del Pezzo surface is a smooth projective
surface with an ample anticanonical class.
For a del Pezzo surface $S$, by its degree we mean its (anti)canonical degree~$K_S^2$.
Del Pezzo surfaces of degree $3$ are exactly smooth cubic surfaces.

\smallskip
\textbf{Acknowledgements.}
I am grateful to  A.\,Trepalin and V.\,Vologodsky
for useful discussions, and to J.-L.\,Colliot-Th\'el\`ene
for interesting references.

\section{Birational models of cubic surfaces}
\label{section:cubic}

The following assertion is known as the theorem of Lang and Nishimura.

\begin{theorem}[{see e.g.~\cite[Lemma~1.1]{VA}}]
\label{theorem:Lang-Nishimura}
Let $X$ and $Y$ be smooth projective varieties over an arbitrary field $\KK$.
Suppose that $X$ is birational to $Y$.
Then $X$ has a $\KK$-point of degree at most~$d$
if and only if $Y$ has a $\KK$-point of degree at most~$d$.
\end{theorem}

It appears that a cubic surface without $\KK$-points
cannot be birationally transformed to del Pezzo surfaces of
degrees~$1$, $2$, $4$, $5$, $7$, and~$8$.

\begin{lemma}
\label{lemma:points-deg-2}
Let $S$ be a smooth cubic surface
over a field $\KK$ of characteristic zero
with~\mbox{$S(\KK)=\varnothing$}.
Then $S$ has no points
of degree $2$ over $\KK$.
\end{lemma}

\begin{proof}
Suppose that $S$ contains a point $P$ of degree $2$ over $\KK$.
Let $L$ be the line in $\PP^3$ passing through $P$; then $L$ is defined
over $\KK$ as well. Note that $L$ always has a $\KK$-point.
If $L$ is contained in $S$, then $S$ has a $\KK$-point.
If $L$ is not contained in $S$, then the third intersection point of
$L_{\bar{\KK}}$ with $S_{\bar{\KK}}$ is
$\Gal\left(\bar{\KK}/\KK\right)$-invariant
and thus is defined over $\KK$.
\end{proof}

\begin{corollary}\label{corollary:1-2}
Let $S$ be a smooth cubic surface
over a field $\KK$ of characteristic zero
with~\mbox{$S(\KK)=\varnothing$}.
Then $S$ is not birational to del Pezzo surfaces of degree $1$ and $2$.
\end{corollary}

\begin{proof}
A del Pezzo surface of degree $d$
over $\KK$ always has a $\KK$-point of degree at most $d$.
Thus the assertion follows from Theorem~\ref{theorem:Lang-Nishimura}
and Lemma~\ref{lemma:points-deg-2}.
\end{proof}

\begin{lemma}\label{lemma:dP4}
Let $S$ be a smooth cubic surface
over a field $\KK$ of characteristic zero with~\mbox{$S(\KK)=\varnothing$}.
Then $S$ is not birational to del Pezzo surfaces of degree~$4$.
\end{lemma}

\begin{proof}
Suppose that
$S$ is birational to a del Pezzo surface $S'$ of degree~$4$.
Then $S'$ has no $\KK$-points by Theorem~\ref{theorem:Lang-Nishimura}.
However, $S$ has a point of degree~$3$ over $\KK$, and thus
Theorem~\ref{theorem:Lang-Nishimura} implies that $S'$ also has a point $P$ of
degree~$3$.
In this case it is known that $S'$ has a $\KK$-point, see~\cite{Coray}
(alternatively, this can be deduced from a more general
theorem due to Amer, Brumer, and Leep, see \cite[Theorem~17.14]{EKM}).
In spite of this, we prefer to give an elementary proof.

Let $P_1$, $P_2$, and $P_3$ be the three points of $P_{\bar{\KK}}$.
Consider the anticanonical embedding~\mbox{$S'\hookrightarrow\PP^4$}.
Suppose that the points $P_1$, $P_2$, and $P_3$ are collinear, and
consider the line~$L$
in~$\PP^4_{\bar{\KK}}$ passing through them. Then $L$
is $\Gal\left(\bar{\KK}/\KK\right)$-invariant and thus defined over~$\KK$.
Since it has at least three common points with the surface $S'$,
and $S'$ can be represented as an intersection of two quadrics in $\PP^4$,
we conclude that $L$ is contained in $S'$. Since $L$ has a $\KK$-point,
we see that $S'$ has a $\KK$-point as well, which gives a contradiction.

Therefore, the points $P_1$, $P_2$, and $P_3$ are not collinear.
Let $\Pi$ be the plane in $\PP^4_{\bar{\KK}}$ passing through
these three points. Then $\Pi$
is $\Gal\left(\bar{\KK}/\KK\right)$-invariant and thus defined over $\KK$.
If the intersection $S'_{\bar{\KK}}\cap\Pi$
is zero-dimensional, it is easy to see that
it consists of~$P_1$,~$P_2$,~$P_3$, and one more point $P_4$.
The point $P_4$ is $\Gal\left(\bar{\KK}/\KK\right)$-invariant
and thus defined over~$\KK$.
If $C=S'_{\bar{\KK}}\cap\Pi$ is one-dimensional, then the union of its one-dimensional irreducible components is either a line~$L$,
or a pair of coplanar lines $L_1\cup L_2$,
or an irreducible conic.
In the first case, the line $L$
is defined over $\KK$, and thus contains
a $\KK$-point. In the second case, the point $L_1\cap L_2$ is defined
over $\KK$. In the third case, the conic $C$ is defined over $\KK$ and contains
the point $P$ of degree $3$ over $\KK$, which implies that it also contains
a $\KK$-point. In every possible case, we see that $S'$ has a $\KK$-point,
which gives a contradiction.
\end{proof}

\begin{lemma}\label{lemma:dP5-7}
Let $S$ be a smooth cubic surface
over a field $\KK$ of characteristic zero with~\mbox{$S(\KK)=\varnothing$}.
Then $S$ is not birational to del Pezzo surfaces of degree $5$ and $7$.
\end{lemma}

\begin{proof}
A del Pezzo surface of degree $5$ or $7$
over $\KK$ always has a $\KK$-point,
see for instance~\mbox{\cite[Theorem~2.5]{VA}}.
Thus the assertion follows from Theorem~\ref{theorem:Lang-Nishimura}.
\end{proof}

\begin{lemma}\label{lemma:dP8}
Let $S$ be a smooth cubic surface
over a field $\KK$ of characteristic zero with~\mbox{$S(\KK)=\varnothing$}.
Then $S$ is not birational to a del Pezzo surface of degree $8$.
\end{lemma}

\begin{proof}
Suppose that $S$ is birational to a del Pezzo surface $S'$ of degree $8$.
Then~$S'$ has no points of degree $1$ and $2$ over $\KK$
by Theorem~\ref{theorem:Lang-Nishimura}
and Lemma~\ref{lemma:points-deg-2}.
If $S'_{\bar{\KK}}$ is not isomorphic to~\mbox{$\PP^1\times\PP^1$}, then
$S'$ contains a $(-1)$-curve defined over $\KK$; this implies
that~$S'$ has a $\KK$-point, which gives a contradiction.
Therefore, we see that~\mbox{$S'_{\bar{\KK}}\cong\PP^1\times\PP^1$}.
Hence the surface $S'$ is isomorphic either
to a product~\mbox{$C_1\times C_2$},
where~$C_1$ and~$C_2$ are conics over~$\KK$, or to
the Weil restriction of scalars
$R_{\LL/\KK} Q$ of a conic~$Q$ defined over some quadratic
extension $\LL$ of~$\KK$, see e.g.~\mbox{\cite[Lemma~7.3(i)]{SV}}.

Note that $S$ always has a point of degree $3$ over $\KK$.
By Theorem~\ref{theorem:Lang-Nishimura}, this implies
that~$S'$ has a point of degree $1$ or $3$ over $\KK$.
Since $S'$ has no $\KK$-points, we conclude that it
has a $\KK$-point of degree~$3$.
If $S'\cong C_1\times C_2$, then each of the conics $C_i$ contains a point of
odd degree over $\KK$, and hence $C_1\cong C_2\cong\PP^1$. This means that
$S'$ has a $\KK$-point, which gives a contradiction.
If $S'\cong R_{\LL/\KK} Q$, then the conic~$Q$ contains a point of degree $3$
over $\LL$, and hence is isomorphic to $\PP^1$. Thus $Q$ has
an $\LL$-point, and so~$S'$ has a $\KK$-point, which again gives
a contradiction.
\end{proof}

\begin{remark}
If $S$ is a smooth cubic surface over
a field $\KK$ of characteristic zero with~\mbox{$S(\KK)=\varnothing$},
then $S$ is not birational
to any smooth surface with a point of degree at most~$2$
over~$\KK$ by Lemma~\ref{lemma:points-deg-2}
and Theorem~\ref{theorem:Lang-Nishimura}.
Therefore, in some cases the assertion of
Lemma~\ref{lemma:dP8} can be deduced from the existence of
such a point on the del Pezzo surface~$S'$ of degree~$8$;
for instance, this is always the case when $S'$ is a quadric in $\PP^3$.
\end{remark}

The next assertion follows from the classification
of Sarkisov links of type~IV, see~\mbox{\cite[Theorem~2.6]{Iskovskikh96}}.

\begin{lemma}\label{lemma:Sarkisov-IV}
Let $S$ be a del Pezzo surface of degree different from $1$, $2$, $4$, and $8$
over a
field of characteristic zero such that $\rkPic(S)=2$.
Suppose that one of the two extremal contractions from $S$ is a conic bundle.
Then the other extremal contraction is not a conic bundle.
\end{lemma}

Now we can describe possible birational models of cubic surfaces without
points.

\begin{lemma}[{cf. \cite[Corollary~2.5]{Shramov-SB}}]
\label{lemma:birational-to-SB}
Let $S$ be a smooth cubic surface
over a field $\KK$ of characteristic zero
with~\mbox{$S(\KK)=\varnothing$}.
The following assertions hold.
\begin{itemize}
\item[(i)] One has $\rkPic(S)\le 3$.

\item[(ii)] If $\rkPic(S)=3$, then $S$ is birational
to a non-trivial Severi--Brauer surface.

\item[(iii)] If $\rkPic(S)=2$, then $S$ is birational
either to a non-trivial
Severi--Brauer surface, or to a del Pezzo surface $S'$ of degree $6$
with $\rkPic(S')=1$.
\end{itemize}
\end{lemma}

\begin{proof}
Suppose that $\rkPic(S)\ge 2$.
If $\rkPic(S)\ge 3$, then there is an extremal contraction~\mbox{$S\to S'$}
to a smooth surface $S'$.
If $\rkPic(S)=2$, then there also
exists an extremal contraction~\mbox{$S\to S'$} to a smooth surface $S'$ by
Lemma~\ref{lemma:Sarkisov-IV}.
In both cases $S'$ is a
a del Pezzo surface of degree $d'>3$ with
$$
\rkPic(S')=\rkPic(S)-1.
$$
Note that $S'$ has no $\KK$-points by Theorem~\ref{theorem:Lang-Nishimura}.
Furthermore, by Lemmas~\ref{lemma:dP4}, \ref{lemma:dP5-7}, and~\ref{lemma:dP8},
one has $d'=6$ or $d'=9$.
In particular, this proves
assertion~(iii).

Now suppose that $\rkPic(S)\ge 3$. Then $\rkPic(S')\ge 2$,
and hence $d'$ cannot be equal to~$9$;
so, we have $d'=6$.
As before, it follows from Lemma~\ref{lemma:Sarkisov-IV}
that there exists an extremal contraction $S'\to S''$
to a del Pezzo surface~$S''$ of degree $d''>6$
with
$$
\rkPic(S'')=\rkPic(S')-1.
$$
Applying Lemmas~\ref{lemma:dP5-7} and~\ref{lemma:dP8},
we conclude that $d''$ cannot be equal to $7$ and $8$.
Therefore, one has~\mbox{$d''=9$}, so that $S''$ is a Severi--Brauer surface.
In particular, one has $\rkPic(S'')=1$, and hence~\mbox{$\rkPic(S)=3$}.
This proves assertions~(i) and~(ii).
\end{proof}

We will complete a classification
of birational models of cubic surfaces without points
in Section~\ref{section:dP6}, where we prove
Proposition~\ref{proposition:all-birational-models}.

For the next result we refer the reader to
\cite[Theorem~V.5.1]{Manin-CubicForms}
and \cite[Theorem~1.5.6]{Cheltsov-BR};
it can be also deduced from \cite[Theorem~2.6]{Iskovskikh96}.

\begin{theorem}\label{theorem:minimal-cubic}
Let $S$ be a smooth cubic surface
over a field $\KK$ of characteristic zero
such that $\rkPic(S)=1$. Let $S'$ be a smooth surface over $\KK$
such that~$S'$ is either a del Pezzo
surface with $\rkPic(S')=1$, or a conic bundle with $\rkPic(S')=2$.
Suppose that $S'$ is birational to $S$.
Then $S'\cong S$. Moreover, if $S(\KK)=\varnothing$, then
$\Bir(S)=\Aut(S)$.
\end{theorem}

\begin{corollary}\label{corollary:minimal-cubic}
Let $S$ be a smooth cubic surface
over a field $\KK$ of characteristic zero
such that $S(\KK)=\varnothing$ and $\rkPic(S)=1$.
Let $S'$ be a smooth surface over $\KK$
such that~$S'$ is either a del Pezzo
surface or a conic bundle.
Suppose that $S'$ is birational to $S$.
Then~\mbox{$S'\cong S$}. In particular, $S'$ is not a conic bundle.
\end{corollary}

\begin{proof}
If there is a conic bundle structure
$S'\to C$, we can perform extremal contractions
over $C$ if necessary and assume that $\rkPic(S')=2$; this is impossible
by Theorem~\ref{theorem:minimal-cubic}.
Therefore, we will assume that $S'$ is a del Pezzo surface. Let $d'$
be the degree of $S'$.
Note that~\mbox{$d'\ge 3$} by Corollary~\ref{corollary:1-2}.
Suppose that either $d'>3$, or $d'=3$ and $\rkPic(S')>1$.
We can perform several extremal contractions if necessary
and obtain a del Pezzo surface~$S''$ of degree $d''>3$ and
$\rkPic(S'')=1$.
The surface $S''$ is birational to $S$, which is again impossible by
Theorem~\ref{theorem:minimal-cubic}.
\end{proof}

The following result is a partial generalization of
Theorem~\ref{theorem:minimal-cubic} to the equivariant setting;
it is well known to experts.

\begin{lemma}
\label{lemma:BR}
Let $S$ be a smooth cubic surface
over a field $\KK$ of characteristic zero
with~\mbox{$S(\KK)=\varnothing$}.
Let $G$ be a  group acting on
$S$ such that $\rkPic(S)^G=1$.
Let $S'$ be a del Pezzo surface
over $\KK$ with $\rkPic(S')^G=1$,
and let $\xi\colon S\dashrightarrow S'$ be a $G$-equivariant
birational map.
Then $\xi$ is an isomorphism.
\end{lemma}

\begin{proof}
Suppose that $\xi$ is not an isomorphism.
Choose a very ample $G$-invariant
linear system $\LLL'$ on $S'$, and let $\LLL$ be its proper transform
on $S$. Then $\LLL$ is a mobile non-empty
$G$-invariant linear system on $S$. Since $\rkPic(S)^G=1$, we can write
$$
\LLL\sim_{\QQ} -\theta K_S
$$
for some positive rational number $\theta$.
It follows from the Noether--Fano inequalities
(see~\mbox{\cite[Lemma~2.4(ii)]{Iskovskikh96}})
that one has $\mult_{P}(\LLL)>\theta$ for some point $P$ on $S$.
Let~$r$ be the degree of $P$, and let
$L_1$ and $L_2$ be two general members of the linear system~$\LLL$ (defined over~$\bar{\KK}$).
We see that
$$
3\theta^2=L_1\cdot L_2\ge r\mult_{P}(\LLL)^2>r\theta^2,
$$
and thus $r\le 2$. On the other hand, $S$ has no points
of degree $1$ and $2$ over $\KK$ by Lemma~\ref{lemma:points-deg-2}.
The obtained contradiction shows that $\xi$ is an isomorphism.
\end{proof}

\section{Main example}
\label{section:main-example}

In this section we provide a construction that proves
Theorem~\ref{theorem:main}.

Let $\Bbbk$ be an algebraically closed field of characteristic zero, and let
$\KK=\Bbbk(\lambda,\mu)$, where~$\lambda$ and $\mu$ are
transcendental variables.
Consider the cubic surface $S$ over $\KK$ given in
the projective space $\PP^3$ with homogeneous coordinates
$x$, $y$, $z$, and $t$ by the equation
\begin{equation*}
\lambda x^3+\lambda^2 y^3+\mu z^3+\mu^2 t^3=0.
\end{equation*}
It is straightforward to check that $S$ has no $\KK$-points.
Let $\omega$ be a non-trivial cubic root of unity.
Consider the linear transformations
\begin{equation*}
\begin{aligned}
&\sigma_1\colon (x:y:z:t)\mapsto (\omega x:y:z:t),\\
&\sigma_2\colon (x:y:z:t)\mapsto (x:\omega y:z:t),\\
&\sigma_3\colon (x:y:z:t)\mapsto (x:y:\omega z:t).
\end{aligned}
\end{equation*}
Then $\sigma_1$, $\sigma_2$, and $\sigma_3$ generate a subgroup $G$ of
$\Aut(S)$ isomorphic to $\mumu_3^3$.

Recall that the action of the groups~\mbox{$\Aut(S)$}
and $\Gal\left(\bar{\KK}/\KK\right)$
on $(-1)$-curves on $S_{\bar{\KK}}$ defines homomorphisms
of these groups to the Weyl group~$\WE$. Furthermore, the homomorphism
$\Aut(S)\to\WE$ is an embedding, see~\cite[Corollary~8.2.40]{Dolgachev}.
Thus we can identify the group $\Aut(S)$ with its image in $\WE$.
In the notation of \cite{Carter},
the elements $\sigma_i$, $1\le i\le 3$,
have type~$\mathrm{A}_2^3$; this follows from the
fact that the quotient of $S_{\bar{\KK}}$ by the group
$G_i\cong \mumu_3$ generated
by the element $\sigma_i$ is isomorphic to~$\PP^2$,
and thus~\mbox{$\rkPic(S_{\bar{\KK}})^{G_i}=1$}.

Let $\gamma_{\lambda}$ and $\gamma_{\mu}$ be the images in $\WE$
of generators of the Galois
groups~\mbox{$\Gal\left(\KK\left(\sqrt[3]{\lambda}\right)/\KK\right)$}
and
$\Gal\left(\KK\left(\sqrt[3]{\mu}\right)/\KK\right)$.
Then $\gamma_{\lambda}$ and $\gamma_{\mu}$ are elements of order $3$ that
commute with each other.
Furthermore,
the $27$ lines on $S_{\bar{\KK}}$ are given by equations
$$
\begin{aligned}
\sqrt[3]{\lambda}x+\omega^i\sqrt[3]{\lambda^2}y=
\sqrt[3]{\mu}z+\omega^j\sqrt[3]{\mu^2}t=0,& \\
\sqrt[3]{\lambda}x+\omega^i\sqrt[3]{\mu}z=
\sqrt[3]{\lambda^2}y+\omega^j\sqrt[3]{\mu^2}t=0,& \\
\sqrt[3]{\lambda}x+\omega^i\sqrt[3]{\mu^2}t=
\sqrt[3]{\lambda^2}y+\omega^j\sqrt[3]{\mu}z=0,& \\
i,j\in\{0,1,2\}.&
\end{aligned}
$$
Thus each of the~$27$ lines on $S_{\bar{\KK}}$
is defined over the field $\KK\left(\sqrt[3]{\lambda},\sqrt[3]{\mu}\right)$.
Hence  the image of the whole Galois group $\Gal\left(\bar{\KK}/\KK\right)$
in the Weyl group $\WE$ is the group $\Gamma\cong\mumu_3^2$
generated by~$\gamma_{\lambda}$ and~$\gamma_{\mu}$.

Note that the group $\Gamma$ can be generated by the elements
$$
\gamma_1=\gamma_{\lambda}\gamma_{\mu} \quad \text{and} \quad
\gamma_2=\gamma_{\lambda}\gamma_{\mu}^{-1}.
$$
The element $\gamma_1$ has nine invariant lines on $S_{\bar{\KK}}$, and the
remaining $18$ lines
split into a union of six $\gamma_1$-invariant
triples of pairwise disjoint lines. This means
that $\gamma_1$ has type~$\mathrm{A}_2$ in the notation of~\cite{Carter};
indeed, the other options for an element of order $3$ are
types~$\mathrm{A}_2^2$ and~$\mathrm{A}_2^3$,
but elements of any of these two types
have no invariant lines on~$S_{\bar{\KK}}$.
Similarly, the element $\gamma_2$ also has type~$\mathrm{A}_2$.
Therefore, for $i=1,2$ the subspace $V_i$ of the $\mathbb{C}$-vector space
$$
V=\Pic(S_{\bar{\KK}})\otimes\mathbb{C}
$$
spanned by the eigen-vectors of $\gamma_i$
with non-trivial eigen-values has dimension~$2$.
Denote by~\mbox{$\langle V_1, V_2\rangle$}
the subspace in $V$ generated by $V_1$ and $V_2$.
Since $\Gamma$ is an abelian group generated by $\gamma_{\lambda}$
and $\gamma_{\mu}$, we can bound the rank of the $\Gamma$-invariant
sublattice $\Pic(S_{\bar{\KK}})^{\Gamma}$ of~\mbox{$\Pic(S_{\bar{\KK}})$}
as
$$
\rkPic(S_{\bar{\KK}})^{\Gamma}=
\dim V-\dim \langle V_1, V_2\rangle \ge
\dim V-\dim V_1-\dim V_2=7-2-2=3,
$$
so that
$$
\rkPic(S)=\rkPic(S_{\bar{\KK}})^{\Gamma}\ge 3.
$$
In particular, $S$ is birational
to some non-trivial Severi--Brauer surface $S_0$
by Lemma~\ref{lemma:birational-to-SB}(i),(ii),
and $G$ is embedded as a subgroup into~\mbox{$\Bir(S_0)$}.

Let $G'$ denote the subgroup of $G$ generated by
$\sigma_3$, and let $G''$ denote the subgroup generated by
$\sigma_2$ and $\sigma_3$, so that $G'\cong\mumu_3$ and $G''\cong\mumu_3^2$.
Since the element $\sigma_3$
has type~$\mathrm{A}_2^3$, one has
$$
\rkPic(S)^{G'}=1.
$$
According to Lemma~\ref{lemma:BR}, this implies there is no
$G'$-equivariant birational map from $S$ to $S_0$.
Therefore, $G'$ (and thus also $G''$ and $G$) is not conjugate in
$\Bir(S_0)$ to a subgroup of $\Aut(S_0)$.
This completes the proof of Theorem~\ref{theorem:main}.

\section{Plane cubics}
\label{section:plane}

In this section we discuss the action of the automorphism group of a
plane cubic curve without
points over the base field on the inflection points of the curve.

Recall that every smooth plane cubic curve over a field $\KK$
of characteristic zero has exactly
$9$ inflection points over $\bar{\KK}$. The set $\Sigma\subset\PP^2$ of these
points is defined over $\KK$. On the other
hand, the set~$\Sigma_{\bar{\KK}}$ together with the lines passing through
its points forms a configuration~\mbox{$(9_4,12_3)$}
isomorphic to the configuration of points and lines
on the affine plane $\mathbb{A}^2$ over the field~$\mathbf{F}_3$
of three elements.
Thus, the automorphism group of the configuration $\Sigma_{\bar{\KK}}$
is isomorphic to the group
$$
\GAff\cong \mumu_3^2\rtimes\GL_2(\mathbf{F}_3)
$$
of order $432=16\cdot 27$;
we refer the reader to~\cite[\S7.3]{BK} for more details.

\begin{lemma}\label{lemma:GAff}
Let $H\subset\GAff$ be a subgroup acting on
$\mathbb{A}^2(\mathbf{F}_3)$ without fixed points.
Then $H$ contains an element of order $3$.
\end{lemma}

\begin{proof}
Suppose that $H$ has no elements of order $3$.
Then the order of $H$ is a power of~$2$.
Since $H$ has no fixed points on $\mathbb{A}^2(\mathbf{F}_3)$,
we conclude that $\mathbb{A}^2(\mathbf{F}_3)$ is a disjoint union of
$H$-orbits of even order. This is impossible
because~\mbox{$|\mathbb{A}^2(\mathbf{F}_3)|=9$}.
\end{proof}

\begin{corollary}\label{corollary:GAff}
Let $C$ be a smooth plane cubic curve over a field
$\KK$ of characteristic zero, and let $\Sigma\subset\PP^2$
be the set of its inflection points. Suppose that
$\Sigma(\KK)=\varnothing$.
Then the image of the Galois group $\Gal\left(\bar{\KK}/\KK\right)$
in the group
$\GAff\cong\Aut(\Sigma_{\bar{\KK}})$ contains an element of order~$3$.
\end{corollary}

\section{Automorphism groups}
\label{section:Jordan}

In this section we
prove the first assertion of Theorem~\ref{theorem:cubic}.

\begin{lemma}
\label{lemma:order-3}
Let $S$ be a smooth cubic surface over a field $\KK$ of characteristic zero
with~\mbox{$S(\KK)=\varnothing$}.
Then the order of the group $\Aut(S)$ is odd.
\end{lemma}

\begin{proof}
The argument is identical to the one used in the proof
of \cite[Lemma~3.3]{Shramov-SB}.
Suppose that the order of $\Aut(S)$ is even. Then there exists an automorphism $g$ of $S$
of order $2$. The action of $g$ on $S_{\bar{\KK}}$ can be of one of the two types listed in
\cite[Table~2]{Trepalin-cubic}; in the notation of \cite[Table~2]{Trepalin-cubic} these are
types~1 and~2. If $g$ is of type $1$, then the fixed point locus
$\Fix_{S_{\bar{\KK}}}(g)$ of $g$ on $S_{\bar{\KK}}$
consists of a smooth elliptic curve and one isolated point;
if $g$ is of type $2$, then $\Fix_{S_{\bar{\KK}}}(g)$ consists of a $(-1)$-curve and three isolated points. Note that a $\Gal\left(\bar{\KK}/\KK\right)$-invariant
$(-1)$-curve on $S$ always contains a $\KK$-point, since such a curve is a line
in the anticanonical embedding of $S$. Therefore, in each of these two cases we
find a $\Gal\left(\bar{\KK}/\KK\right)$-invariant point on $S_{\bar{\KK}}$, which gives a contradiction.
\end{proof}

Knowing that a group acting on a cubic surface has odd order
provides strong restrictions on its structure.
In particular, Lemma~\ref{lemma:order-3} implies the following.

\begin{corollary}\label{corollary:S5}
Let $S$ be a smooth cubic surface
over a field $\KK$ of characteristic zero
with~\mbox{$S(\KK)=\varnothing$}.
Suppose that $S_{\bar{\KK}}$ is of one of the types~$\mathrm{II}$,
$\mathrm{V}$, $\mathrm{VI}$, or~$\mathrm{VIII}$ in the notation of
\cite[Table~9.6]{Dolgachev}.
Then the group $\Aut(S)$ is cyclic.
\end{corollary}

\begin{proof}
According to \cite[Table~9.6]{Dolgachev}, one has
$$
\Aut(S_{\bar{\KK}})\subset\SSS_5.
$$
On the other hand, the order of $\Aut(S)$ is odd
by Lemma~\ref{lemma:order-3}, and thus $\Aut(S)$ is either a trivial group, or a cyclic
group of order $3$ or $5$ (note that $\SSS_5$ does not contain
groups of order~$15$).
\end{proof}

Let $\HHH_3$ denote the Heisenberg group of order $27$; this is the only non-abelian
group of order $27$ and exponent $3$. Its center $\z(\HHH_3)$ is isomorphic to $\mumu_3$,
and there is a non-split exact sequence
$$
1\to\z(\HHH_3)\to\HHH_3\to\mumu_3^2\to 1.
$$
On the other hand, there is an isomorphism
$$
\HHH_3\cong\mumu_3^2\rtimes\mumu_3.
$$

\begin{lemma}\label{lemma:H3}
Let $S$ be a smooth cubic surface
over a field $\KK$ of characteristic zero
with~\mbox{$S(\KK)=\varnothing$}.
Suppose that $S_{\bar{\KK}}$ is of one of the
types~$\mathrm{I}$, $\mathrm{III}$, or~$\mathrm{IV}$
in the notation of~\mbox{\cite[Table~9.6]{Dolgachev}}.
Then the group $\Aut(S)$ is abelian.
\end{lemma}

\begin{proof}
The order of $\Aut(S)$ is odd
by Lemma~\ref{lemma:order-3}.
Therefore, according to \cite[Table~9.6]{Dolgachev},
the group $\Aut(S)$ is a subgroup
of the $3$-Sylow subgroup $\Theta\subset\Aut(S_{\bar{\KK}})$;
one has~\mbox{$\Theta\cong\mumu_3^3\rtimes\mumu_3$}
if $S_{\bar{\KK}}$ is of type~I,
and $\Theta\cong\HHH_3$ if $S_{\bar{\KK}}$ is of type~III or~IV.
In either case the central subgroup $\Xi$ of $\Theta$ is isomorphic to
$\mumu_3$.
Note that any subgroup of $\Theta$ that does not contain $\Xi$
is abelian.
Thus we may assume that $\Aut(S)$ contains $\Xi$.
Then~\mbox{$S/\Xi\cong\PP^2$}, and the quotient map
$$
\nu\colon S\to \PP^2
$$
is defined over $\KK$. The branch curve of $\nu$ is a smooth
plane cubic~$C$. Since~$S$ has no $\KK$-points, the curve
$C$ also has no $\KK$-points. Therefore,
the image of the Galois group~\mbox{$\Gal\left(\bar{\KK}/\KK\right)$}
in the automorphism group
$\GAff$ of the configuration of the inflection points of $C_{\bar{\KK}}$
contains an element of order~$3$ by Corollary~\ref{corollary:GAff}.
Since the map $\nu$ is $\Gal\left(\bar{\KK}/\KK\right)$-equivariant, we see that
the image of $\Gal\left(\bar{\KK}/\KK\right)$ in the Weyl group $\WE$
contains an element $\gamma$ of order $3$ or $9$
such that $\gamma$ is not contained in the image of
the group~$\Xi$ in~$\WE$. On the other hand, we know that
the group $\Aut(S)$ must commute with $\gamma$.
Recall that the $3$-Sylow subgroup $\Delta$ of
$\WE$ is isomorphic to $\mumu_3^3\rtimes\mumu_3$.
The center of~$\Delta$ is isomorphic to $\mumu_3$, and according to our
assumptions it coincides with the image in the Weyl group~$\WE$
of the subgroup~\mbox{$\Xi\subset\Aut(S)$}.
Therefore, $\Aut(S)$ is isomorphic to a subgroup of $\Delta$ commuting with
a non-central element~\mbox{$\gamma\in\Delta$}.
It is straightforward to check that
any such subgroup of $\Delta$ is abelian.
\end{proof}

The following example shows that in the notation of the proof
of Lemma~\ref{lemma:H3} the subgroup in $\Aut(\PP^2)$ preserving
the plane cubic $C$ may be non-abelian, although the group~\mbox{$\Aut(S)$}
is abelian.

\begin{example}
Let $\Bbbk$ be an algebraically closed field of characteristic zero,
and let~\mbox{$\KK=\Bbbk(\lambda, \mu)$},
where $\lambda$ and $\mu$ are transcendental variables.
Consider the cubic surface~$S$ over $\KK$ given
by the equation
$$
\mu t^3=x^3+\lambda y^3+\lambda^2 z^3;
$$
in the notation of \cite[Table~9.6]{Dolgachev},
the surface $S_{\bar{\KK}}$ has type~I.
It is easy to see that~$S$ has no $\KK$-points.
Let $\omega$ be a non-trivial cubic root of unity. Let
$\sigma_1$, $\sigma_2$, and~$\sigma_3$ be the automorphisms of $S$
defined as
\begin{equation*}
\begin{aligned}
&\sigma_1\colon (x:y:z:t)\mapsto (\omega x:y:z:t),\\
&\sigma_2\colon (x:y:z:t)\mapsto (x:\omega y:z:t),\\
&\sigma_3\colon (x:y:z:t)\mapsto (x:y:z:\omega t).
\end{aligned}
\end{equation*}
Let $\nu\colon S\to\PP^2$
be the quotient by the group generated by $\sigma_3$.
Denote by $C$ the branch curve of $\nu$.
Then $C$ is a smooth plane cubic
given in $\PP^2$ with homogeneous coordinates~$x$,~$y$, and~$z$
by the equation
$$
x^3+\lambda y^3+\lambda^2 z^3=0.
$$
Let $\omega$ be a non-trivial cubic root of unity.
The group $G\cong \mumu_3^3$ generated by $\sigma_1$,
$\sigma_2$, and~$\sigma_3$ acts on $S$, and its quotient
$\hat{G}\cong\mumu_3^2$ generated by
the linear transformations
$$
\begin{aligned}
&\hat{\sigma}_1\colon (x:y:z)\mapsto (\omega x:y:z),\\
&\hat{\sigma}_2\colon (x:y:z)\mapsto (x:\omega y:z),
\end{aligned}
$$
acts on $\PP^2$ so that the map $\nu$ is $G$-equivariant.
The group $\hat{G}$ preserves the plane curve $C$, but~$C$ actually
has more symmetries.
Namely, it is also preserved by the
linear transformation
$$
\hat{\sigma}\colon (x:y:z)\mapsto (\lambda z:x:y).
$$
Note that $\hat{\sigma}$ lifts to the automorphism
$$
\sigma\colon (x:y:z:t)\mapsto \left(\lambda z:x:y:\sqrt[3]{\lambda}t\right)
$$
of the surface $S_{\bar{\KK}}$.
However, we can see from Lemma~\ref{lemma:H3}
(or check in a straightforward way) that $\hat{\sigma}$ cannot be lifted
to an automorphism of~$S$.
\end{example}

Now we are ready to prove assertion~(i) of
Theorem~\ref{theorem:cubic}.

\begin{lemma}
\label{lemma:cubic-Aut}
Let $\KK$ be a field of characteristic zero, and let $S$ be a smooth cubic
surface over $\KK$. Suppose that $S$ has no $\KK$-points.
Then the group $\Aut(S)$ is abelian.
\end{lemma}

\begin{proof}
Inspecting \cite[Table~9.6]{Dolgachev} which lists all possible non-trivial 
automorphism groups of smooth cubic surfaces over fields of characteristic
zero, we see that the group~\mbox{$\Aut(S_{\bar{\KK}})$} is
abelian if $S_{\bar{\KK}}$ is of one of the types~VII, IX, X, or~XI;
the group~$\Aut(S)$ is
cyclic by Corollary~\ref{corollary:S5}, if
$S_{\bar{\KK}}$ is of one of the types~II, V, VI, or~VIII;
and $\Aut(S)$ is abelian by Lemma~\ref{lemma:H3}, if
$S_{\bar{\KK}}$ is of one of the types~I, III, or~IV.
\end{proof}

\begin{remark}
If $S$ is a smooth cubic surface
over a field $\KK$ of characteristic zero such that~$S_{\bar{\KK}}$ is of one
of the types~VII, IX, X, or~XI, then $S$
has a $\KK$-point. Indeed,
if $S_{\bar{\KK}}$ is of one
of the types~VII, IX, or~XI, then $S_{\bar{\KK}}$ has a unique
Eckardt point, see~\cite[Table~9.6]{Dolgachev};
note that there is a typo for type~VII in~\cite[Table~9.6]{Dolgachev},
cf. the relation between Eckardt points and certain involutions of cubics
described in \cite[Proposition~9.1.23]{Dolgachev}.
Thus $S$ has a $\KK$-point.
If $S_{\bar{\KK}}$ is of type~X, then $S_{\bar{\KK}}$ has exactly two
Eckardt points, and thus $S$ has a $\KK$-point by
Lemma~\ref{lemma:points-deg-2}.
This means that these cases do not actually arise in the
proof of Lemma~\ref{lemma:cubic-Aut}.
\end{remark}

\section{Del Pezzo surfaces of degree $6$}
\label{section:dP6}

In this section we make some observations on del Pezzo surfaces of degree $6$.

\begin{lemma}
\label{lemma:dP6-birational}
Let $S$ be a del Pezzo surface of degree $6$
over a field $\KK$ of characteristic zero
such that $S(\KK)=\varnothing$ and $\rkPic(S)=1$.
Let $S'$ be a smooth surface over~$\KK$
such that~$S'$ is either a del Pezzo
surface with $\rkPic(S')=1$, or a conic bundle with~\mbox{$\rkPic(S')=2$}.
Suppose that $S'$ is birational to $S$.
Then $S'$ is a del Pezzo surface of degree $6$ with~\mbox{$\rkPic(S')=1$}.
\end{lemma}

\begin{proof}
Choose a birational map $S\dasharrow S'$, and decompose it into a sequence
of Sarkisov links
$$
S\stackrel{\chi_1}\dasharrow S_1\stackrel{\chi_2}\dasharrow\ldots
\stackrel{\chi_n}\dasharrow S_n=S'.
$$
We know from \cite[Theorem~2.6]{Iskovskikh96} that
$S_1$ is again a del Pezzo surface of degree $6$ with~\mbox{$\rkPic(S_1)=1$}.
Moreover, one has $S_i(\KK)=\varnothing$
by Theorem~\ref{theorem:Lang-Nishimura}.
Thus, the required assertion follows by
induction on~$n$.
\end{proof}

\begin{corollary}\label{corollary:dP6-birational}
Let $S$ be a del Pezzo surface of degree $6$
over a field $\KK$ of characteristic zero
such that $\rkPic(S)=1$.
Suppose that $S$ has no points of degree $1$ and $2$ over $\KK$.
Let~$S'$ be a smooth surface over $\KK$
such that $S'$ is either a del Pezzo
surface or a conic bundle.
Suppose that $S'$ is birational to $S$.
Then $S'$ is a del Pezzo surface of degree $3$ or~$6$;
in the latter case one has $\rkPic(S')=1$. Furthermore, $S'$ is not
a conic bundle.
\end{corollary}

\begin{proof}
If there is a conic bundle structure
$S'\to C$, we can perform extremal contractions
over $C$ if necessary and assume that $\rkPic(S')=2$; this is impossible
by Lemma~\ref{lemma:dP6-birational}.
Therefore, we will assume that $S'$ is a del Pezzo surface. Let $d'$
be the degree of $S'$.
Note that $S'$ does not contain points of degree $1$ and $2$ over $\KK$
by Theorem~\ref{theorem:Lang-Nishimura}.
In particular, one has~\mbox{$d'>2$}, because
$S'$ always has a $\KK$-point of degree
at most~$d'$.
Similarly, $d'\neq 5$ because a del Pezzo surface of degree $5$ always has
a $\KK$-point.

Suppose that either $d'>6$, or $d'=6$ and $\rkPic(S')>1$.
We can perform several extremal contractions if necessary
and obtain a del Pezzo surface~$S''$ of degree $d''>6$
and~\mbox{$\rkPic(S'')=1$}.
The surface $S''$ is birational to $S$, which is impossible by
Lemma~\ref{lemma:dP6-birational}.

Now suppose that $d'=4$. By Lemma~\ref{lemma:dP6-birational}
one has $\rkPic(S')>1$.  Performing several extremal contractions, we
obtain a del Pezzo surface $S''$ of degree $d''>d'$
and~\mbox{$\rkPic(S'')=1$}.
Again applying Lemma~\ref{lemma:dP6-birational}, we see that $d''=6$.
Thus the morphism $S'\to S''$ contracts two $(-1)$-curves over $\bar{\KK}$,
and so $S''$ contains a $\KK$-point of degree at most $2$ over $\KK$.
This is impossible by Theorem~\ref{theorem:Lang-Nishimura}.

Therefore, the only possible cases are $d'=3$ and $d'=6$,
and in the latter case one has~\mbox{$\rkPic(S')=1$}.
\end{proof}

Let $S$ be a del Pezzo surface of degree $6$ over a field $\KK$ of characteristic
zero. Recall from \cite[Theorem~8.4.2]{Dolgachev} that
\begin{equation}\label{eq:dP6-Aut}
\Aut(S_{\bar{\KK}})\cong \left(\bar{\KK}^*\right)^2\rtimes
(\SSS_3\times\mumu_2).
\end{equation}
The configuration of $(-1)$-curves on $S_{\bar{\KK}}$ can be interpreted as
a hexagon, and the group~\mbox{$\SSS_3\times\mumu_2$} on the right hand side
of~\eqref{eq:dP6-Aut} can be identified with the automorphism
group of this configuration. Note that there
are two possible ways to choose the factor~$\SSS_3$
in~\mbox{$\SSS_3\times\mumu_2$};
we will always assume that this subgroup is chosen so that it preserves
the two triples of pairwise disjoint $(-1)$-curves.

\begin{lemma}\label{lemma:dP6-involution}
Let $S$ be a del Pezzo surface of degree $6$
over a field $\KK$ of characteristic zero.
Suppose that $S$ has no points of degree $1$ and $2$ over $\KK$.
Then the image of the group~\mbox{$\Aut(S)$} in $\SSS_3\times\mumu_2$
does not contain the central element $z$ of $\SSS_3\times\mumu_2$.
\end{lemma}

\begin{proof}
Suppose that the image of $\Aut(S)$ contains $z$.
Let $\tilde{z}$ be an element of $\Aut(S)$ mapped to~$z$.

There exists an extension $\LL/\KK$ of degree at most $2$ such that
the image of the Galois group~\mbox{$\Gal\left(\bar{\KK}/\LL\right)$}
in~\mbox{$\SSS_3\times\mumu_2$}
is contained in the factor $\SSS_3$ of $\SSS_3\times\mumu_2$.
Thus~\mbox{$\Gal\left(\bar{\KK}/\LL\right)$} preserves the two triples of
pairwise disjoint $(-1)$-curves on~$S_{\bar{\KK}}$. This means that
the surface $S_{\LL}$ has two contractions
$$
\pi_i\colon S_{\LL}\to B_i, \quad i=1,2,
$$
where $B_i$ are Severi--Brauer surfaces
over $\LL$. Consider the birational map
$$
\pi_2\circ\pi_1^{-1}\colon B_1\dasharrow B_2.
$$
After extension of scalars to $\bar{\KK}$ it defines a birational
map $\PP^2\dasharrow\PP^2$ that
is given by a certain linear system of conics
(see \cite[\S2]{Shramov-SB} and references therein).
This implies that $B_1$ and $B_2$ correspond to opposite central simple
algebras, see e.g. \cite[Exercise~3.3.7(iii)]{GS}.
On the other hand, the element $\tilde{z}\in\Aut(S_{\LL})$
swaps the contractions $\pi_1$ and $\pi_2$, which means that
$B_1\cong B_2$.
Therefore, both $B_1$ and $B_2$ are isomorphic to $\PP^2$.
In particular, the surface $S_{\LL}$ has an $\LL$-point. Hence
$S$ has a $\KK$-point of degree at most $2$, which contradicts our assumptions.
\end{proof}

\begin{lemma}\label{lemma:dP6-Aut}
Let $S$ be a del Pezzo surface of degree $6$
over a field $\KK$ of characteristic zero
such that $\rkPic(S)=1$.
Suppose that $S$ has no points of degree $1$ and $2$ over $\KK$.
Then the group $\Aut(S)$ has a normal abelian subgroup of index at most~$3$.
\end{lemma}

\begin{proof}
Let $\Gamma$ be the image of the group
$\Gal\left(\bar{\KK}/\KK\right)$ in $\SSS_3\times\mumu_2$.
Since $\rkPic(S)=1$, we conclude that $\Gamma$ contains
an element $\gamma$ of order $3$.
The image $\Delta$ of $\Aut(S)$ in $\SSS_3\times\mumu_2$
commutes with $\Gamma$, and thus it is contained
in the centralizer of $\gamma$, which is the group of order $6$ generated
by $\gamma$ and the
central element $z$ of $\SSS_3\times\mumu_2$.
Also, we know from Lemma~\ref{lemma:dP6-involution}
that $\Delta$ does not contain the
element $z$.
Hence $\Delta$ is contained in the group of order $3$ generated by $\gamma$,
and so the
kernel of the homomorphism $\Aut(S)\to \SSS_3\times\mumu_2$ is an abelian group
of index at most $3$ in $\Aut(S)$.
\end{proof}

\begin{corollary}\label{corollary:J-dP6}
Let $S$ be a del Pezzo surface of degree $6$
over a field $\KK$ of characteristic zero
such that $\rkPic(S)=1$.
Suppose that $S$ has no points of degree $1$ and $2$ over $\KK$.
Let $G$ be a finite subgroup of $\Bir(S)$. Then
$G$ has a normal abelian subgroup of index at most~$3$.
\end{corollary}

\begin{proof}
Regularizing the action of $G$ and running a $G$-Minimal Model Program,
we obtain a surface $S'$ with an action of $G$ and a $G$-equivariant
birational map $S\dasharrow S'$,
such that~$S'$ is either a del Pezzo
surface, or a conic bundle.
Note that $S'$ has no points of degree~$1$ and~$2$ over $\KK$
by Theorem~\ref{theorem:Lang-Nishimura}.
Applying Corollary~\ref{corollary:dP6-birational}, we see that
$S'$ is either a cubic surface, or a del Pezzo surface of degree $6$ with
$\rkPic(S')=1$. In the former case the group~$G$ is abelian
by Lemma~\ref{lemma:cubic-Aut}.
In the latter case the assertion follows from Lemma~\ref{lemma:dP6-Aut}.
\end{proof}

We conclude this section by a classification of birational models of
cubic surfaces without points and del Pezzo surfaces of degree $6$
without points of degree $1$ and $2$ over the base field. It is not
necessary for the proof of the main results of this paper, but we find it
interesting on its own.

\begin{lemma}[{cf. Lemma~\ref{lemma:birational-to-SB}}]
\label{lemma:dP6-models}
Let $S$ be a del Pezzo surface of degree $6$
over a field $\KK$ of characteristic zero
such that $S$ has no points of degree $1$ and $2$ over $\KK$.
The following assertions hold.
\begin{itemize}
\item[(i)] One has $\rkPic(S)\le 2$.

\item[(ii)] If $\rkPic(S)=2$, then $S$ is birational
to a non-trivial Severi--Brauer surface.
\end{itemize}
\end{lemma}

\begin{proof}
Suppose that $\rkPic(S)\ge 2$.
Keeping in mind Lemma~\ref{lemma:Sarkisov-IV}, we see that
there exists an extremal contraction~\mbox{$\pi\colon S\to S'$}
to a del Pezzo surface $S'$ of degree $d'>6$ with
$$
\rkPic(S')=\rkPic(S)-1.
$$
The surface $S'$ has no points of degree $1$ and $2$ over $\KK$
by Theorem~\ref{theorem:Lang-Nishimura}.
If $d'\le 8$, then $\pi$ contracts at most two
$(-1)$-curves on $S_{\bar{\KK}}$. This implies that $S'$ has
a $\KK$-point of degree at most $2$, which is a contradiction.
Hence $d'=9$, so that $S'$ is a non-trivial Severi--Brauer surface.
In particular, one has $\rkPic(S')=1$ and $\rkPic(S)=2$.
\end{proof}

We complete this section by a classification of birational models
of smooth cubic surfaces without points and del Pezzo surfaces of degree~$6$
without points of degree~$1$ and~$2$.

\begin{proof}[Proof of Proposition~\ref{proposition:all-birational-models}]
Note that $S$ has no points of degree $1$ and $2$ over $\KK$.
If~$S$ is a del Pezzo surface of degree $6$, this holds by assumption;
if $S$ is a cubic surface, then this holds by Lemma~\ref{lemma:points-deg-2};
if~$S$ is a Severi--Brauer surface, this is well known,
see e.g.~\mbox{\cite[Theorem~53(2)]{Kollar-SB}}. Therefore, by
Theorem~\ref{theorem:Lang-Nishimura} any smooth surface $S'$ birational
to~$S$ has no points of degree $1$ and~$2$ over~$\KK$.

Suppose that $S$ is a non-trivial Severi--Brauer surface.
Then the required
assertion follows from~\mbox{\cite[Corollary~2.4]{Shramov-SB}}.

Suppose that $S$ is a del Pezzo surface of degree $6$ with no points
of degree $1$ and $2$ over $\KK$. Then $\rkPic(S)\le 2$
by Lemma~\ref{lemma:dP6-models}(i).
If $\rkPic(S)=1$, the assertion is given by
Corollary~\ref{corollary:dP6-birational}.
If $\rkPic(S)=2$, we know from Lemma~\ref{lemma:dP6-models}(ii)
that $S$ is birational to a Severi--Brauer surface, and the required assertion
follows.

Finally, suppose that $S$ is a smooth cubic surface with $S(\KK)=\varnothing$.
Then~\mbox{$\rkPic(S)\le 3$} by Lemma~\ref{lemma:birational-to-SB}(i).
If $\rkPic(S)=1$, the assertion is given by
Corollary~\ref{corollary:minimal-cubic}.
If~\mbox{$\rkPic(S)\ge 2$},
we know from Lemma~\ref{lemma:birational-to-SB}(ii),(iii)
that $S$ is birational either to a del Pezzo surface $S'$
of degree $6$ with~\mbox{$\rkPic(S')=1$}, or to a
Severi--Brauer surface.
Therefore, in these cases the required assertion is implied by the
above arguments.
\end{proof}

\section{Birational automorphism groups}
\label{section:Bir}

In this section we prove the remaining assertions
of Theorem~\ref{theorem:cubic}.
Let us start with deriving some consequences from Sections~\ref{section:cubic}
and~\ref{section:dP6}.

\begin{corollary}\label{corollary:J-rkPic-3}
Let $S$ be a smooth cubic surface
over a field $\KK$ of characteristic zero
with~\mbox{$S(\KK)=\varnothing$} and $\rkPic(S)=3$.
Then every finite subgroup of $\Bir(S)$ contains a normal abelian
subgroup of index at most $3$. Moreover,
there exists a field $\KK_0$ of characteristic
zero, a smooth cubic surface $S_0$ over $\KK_0$
with~\mbox{$S_0(\KK_0)=\varnothing$}
and $\rkPic(S_0)=3$,
and a finite subgroup $G_0\subset\Bir(S_0)$ such that
the minimal index of a normal abelian subgroup in $G_0$ equals~$3$.
\end{corollary}

\begin{proof}
We know from Lemma~\ref{lemma:birational-to-SB}(ii) that
$S$ is birational to a non-trivial
Severi--Brauer surface.
Therefore, the first of the required assertions follows from
Theorem~\ref{theorem:SB}.

To prove the second assertion,
take a field $\KK_0$ of characteristic
zero and a non-trivial Severi--Brauer surface $S'_0$ over $\KK_0$ such
that $\Aut(S'_0)$ contains a non-abelian finite subgroup~$G_0$; these exist
by Theorem~\ref{theorem:SB}. Then the minimal index
of a normal abelian subgroup in $G_0$ equals $3$. Now we construct
a cubic surface $S_0$ as a blow up of $S'_0$ at two points of degree $3$
over $\KK_0$ in general position; note that such points always
exist, see e.g.~\mbox{\cite[Theorem~53(3)]{Kollar-SB}}.
Thus $S_0$ is a cubic surface without $\KK_0$-points, and $\Bir(S_0)$
contains the group~$G_0$.
\end{proof}

\begin{corollary}\label{corollary:J-rkPic-2}
Let $S$ be a smooth cubic surface
over a field $\KK$ of characteristic zero
with~\mbox{$S(\KK)=\varnothing$} and $\rkPic(S)=2$.
Then every finite subgroup of $\Bir(S)$ contains a normal abelian
subgroup of index at most~$3$.
\end{corollary}

\begin{proof}
We know from Lemma~\ref{lemma:birational-to-SB}(iii)
that $S$ is birational either to a non-trivial Severi--Brauer surface,
or to a del Pezzo surface
$S'$ of degree~$6$ with $\rkPic(S')=1$.
In the former case the assertion follows from Theorem~\ref{theorem:SB}.
In the latter case we
recall that $S$ does not contain points of degree $1$ and $2$
over~$\KK$ by Lemma~\ref{lemma:points-deg-2}.
Hence $S'$ also does not contain points of degree $1$ and $2$ over
$\KK$ by Theorem~\ref{theorem:Lang-Nishimura}. Therefore, the required
assertion follows from Corollary~\ref{corollary:J-dP6}.
\end{proof}

\begin{corollary}\label{corollary:J-rkPic-1}
Let $S$ be a smooth cubic surface
over a field $\KK$ of characteristic zero
with~\mbox{$S(\KK)=\varnothing$} and $\rkPic(S)=1$.
Then every finite subgroup of $\Bir(S)$ is abelian.
\end{corollary}

\begin{proof}
By Theorem~\ref{theorem:minimal-cubic}, one has $\Bir(S)=\Aut(S)$.
Thus the assertion follows from Lemma~\ref{lemma:cubic-Aut}.
\end{proof}

\begin{remark}
For an alternative proof of Corollary~\ref{corollary:J-rkPic-1} that
does not use the last assertion of Theorem~\ref{theorem:minimal-cubic},
one can argue as follows.
Regularizing the action of $G$ and running a $G$-Minimal Model Program,
we obtain a surface $S'$ with an action of $G$ and a $G$-equivariant
birational map $S\dasharrow S'$,
such that $S'$ is either a del Pezzo
surface, or a conic bundle.
By Corollary~\ref{corollary:minimal-cubic}, we have $S'\cong S$.
Thus the assertion follows from Lemma~\ref{lemma:cubic-Aut}.
\end{remark}

\smallskip
Finally, we complete the proof of Theorem~\ref{theorem:cubic}.
Assertion~(i) is given by Lemma~\ref{lemma:cubic-Aut}.
Furthermore, we know from Lemma~\ref{lemma:birational-to-SB}(i)
that $\rkPic(S)\le 3$. Thus, assertion~(ii) is implied by
Corollaries~\ref{corollary:J-rkPic-3}, \ref{corollary:J-rkPic-2},
and~\ref{corollary:J-rkPic-1} in the cases when
$\rkPic(S)$ equals $3$, $2$, and $1$, respectively.
Assertion~(iii) is given by Corollary~\ref{corollary:J-rkPic-3}.

\end{document}